\documentclass[12pt, a4paper]{amsart}
\usepackage{amscd,amsmath,amssymb,amsthm,amsfonts}
\usepackage[alphabetic]{amsrefs}
\usepackage{mathrsfs}
\usepackage[shortlabels]{enumitem}
\usepackage[all]{xy}

\usepackage{xcolor}
\usepackage[hypertexnames=true, citecolor = green, colorlinks = false, linkcolor = red, linktoc = none, pdffitwindow = false, urlbordercolor = white]{hyperref}%

\def\max{\operatorname{max}}

\def\N{\mathbb{N}}

 \newcommand{\IN}[0]{\mathbb{N}}
 
\newcommand{\IQ}[0]{\mathbb{Q}} \newcommand{\IR}[0]{\mathbb{R}}
 \newcommand{\IT}[0]{\mathbb{T}}

 \newcommand{\IZ}[0]{\mathbb{Z}}

\newcommand{\CA}[0]{\mathcal{A}} \newcommand{\CB}[0]{\mathcal{B}}
 
 \newcommand{\CF}[0]{\mathcal{F}}

\newcommand{\CO}[0]{\mathcal{O}} 
\newcommand{\CQ}[0]{\mathcal{Q}} 
 
\newcommand{\CU}[0]{\mathcal{U}}

\def\gxp{G \rtimes_\theta P}

\newtheorem{thm}{Theorem}[section]

\newtheorem{lemma}[thm]{Lemma}

\newtheorem{proposition}[thm]{Proposition}

\theoremstyle{definition}

\theoremstyle{remark}
\newtheorem{remark}[thm]{Remark}

\newtheorem{example}[thm]{Example}

\numberwithin{equation}{section}

\begin{document}

\title[A Fej\'{e}r theorem for boundary quotients]{A Fej\'{e}r theorem for boundary quotients arising from algebraic dynamical systems}

\author[V.~Aiello]{Valeriano Aiello}
\address{Section de Math\'ematiques,
Universit\'e de Gen\`eve,
2-4 rue du Li\`evre, Case Postale 64, 1211 Gen\`eve 4, Suisse}
\email{valerianoaiello@gmail.com}

\author[R.~Conti]{Roberto Conti}
\address{Dipartimento di Scienze di Base e Applicate per l'Ingegneria, 
Sapienza Universit\`a di Roma, Via A. Scarpa 16, 00161 Roma, Italy}
\email{roberto.conti@sbai.uniroma1.it}

\author[S.~Rossi]{Stefano Rossi}
\address{Dipartimento di Matematica, Universit\`{a} di Roma Tor Vergata, Via della Ricerca Scientifica, 1, I-00133 Roma, Italy}
\email{rossis@mat.uniroma2.it1}

\thanks{Valeriano Aiello is supported by the Swiss National Science Foundation. Roberto Conti is supported by Sapienza Universit\`a di Roma.
Stefano Rossi is supported by European Research Council Advanced Grant 669240 QUEST}

\begin{abstract}
A Fej\'{e}r-type theorem is proved within the framework of $C^*$-algebras associated with certain irreversible algebraic dynamical systems. 
This makes it possible 
to strengthen a result on the structure of the relative commutant of a family of generating isometries in a boundary quotient.
\end{abstract}

\date{\today}
\maketitle

\section{Introduction}

The uniform convergence of the Fourier series $S_n(f)$ of a given continuous function 
$f\in C(\IT)$ is notoriously a delicate matter. Indeed, although strong sufficient conditions are easy to find, for instance that
$f$ is continuously differentiable, optimal conditions are not as easy to spot.
 No wonder, fairly pathological examples of continuous functions can be exhibited whose
Fourier series fails to converge uniformly. As a matter of fact, far more pathological examples can be given for which the Fourier series
behaves so badly as to diverge at some point. This was already known as long ago as $1876$ to du Bois-Reymond, who is generally credited with having produced  the first example of this sort.  
Even so, it is still possible to have uniform convergence provided that one considers the Ces\` aro mean 
of the sequence $\{S_n(f):n\in\IN\}$ instead. In other words, the sequence $\frac{1}{N}\sum_{n=0}^{N-1}S_n(f)$ does
converge uniformly to $f$, which is nothing but the content of a classical theorem proved by Fej\'er in 1904.
As well as being an interesting result in its own right, Fej\'er's theorem has proved to be a good source of inspiration for modern
research too. On the one hand, one may exploit the nexus of ideas and techniques involved in the proof of the classical result
to treat fine properties, such as weak amenability, of locally compact groups, as is done in e.g. \cite{Vale}.
On the other hand, one may also endeavour to generalize the content of the classical theorem to cover a wider variety of situations,
when $C(\IT)$ is replaced by a non-commutative $C^*$-algebra.
The first generalizations considered in this direction \cite{Zel} dealt with  $C^*$-algebras  obtained as crossed products (also see e.g. \cite{Davidson} for a simplified treatment where crossed products are only considered with respect to the action of $\IZ$, and
\cite[Proposition B1]{Raeb} for $C^*$-algebras acted upon by an $n$-dimensional torus).
 Greater generality was then achieved in \cite{ConBed2015, ContiBedos2016}. In recent times, crossed products by  locally compact groups with the approximation property have been addressed in \cite{Neufang} in a von Neumann algebra setting also. 
In this note, though, the focus is on the so-called boundary quotient $C^*$-algebras, which of late have been given a good deal
of attention in a series of papers \cite{BLS1, BS1, ACRS, BLSR19}. 
If we are to take a step further into the analysis of  these $C^*$-algebras, 
suitable approximation properties appear to be the needed tool to make finer computations. 
Therefore, what the note aims to do is bridge this gap by providing a theorem \`a la  Fej\'er, which we believe
may and indeed does come in useful at least to treat a natural subclass of boundary quotient $C^*$-algebras.

\section{Preliminaries and notations}\label{prel}
In this section we rather quickly recall the definitions and the properties of the  $C^*$-algebras 
that we will be treating in the present work. These are obtained out of so-called
algebraic dynamical systems, for which much wider information can be found in \cite{ACRS} and some of the references therein.
The algebraic dynamical systems we are actually interested in are of a rather special kind. Namely, they are a triple $(G, P,\theta)$, where:
\begin{enumerate}[(a)]
\item $G$ is a countable discrete group;
\item $P$ is  a discrete (countable)  abelian cancellative unital semigroup 
such that the intersection of two  principal  
  ideals is still a (possibly empty) principal 
 ideal; and
\item $\theta$ is an action of $P$ upon $G$ through injective homomorphisms that  \emph{preserve the order}, that is $pP \cap qP = rP$ implies $\theta_p(G) \cap \theta_q(G) = \theta_r(G)$.
\end{enumerate}
Furthermore, $P$ can be made into a directed set by introducing the order relation $p \geq q \Leftrightarrow p \in qP$.
We denote by $\CQ(\gxp)$ the associated boundary quotient $C^*$-algebra, which is by definition the universal $C^*$-algebra generated by a unitary representation $u$ of the group $G$ and a representation $s$ of the semigroup $P$ by isometries satisfying the relations
\begin{enumerate}[(I)]
\item $s_pu_g = u_{\theta_p(g)}s_p$,
\item $s_p^*u_g^{\phantom{*}}s_q = \begin{cases} u_{g_1}^{\phantom{*}}s_{p'}^{\phantom{*}}s_{q'}^*u_{g_2} &,\text{ if } g = \theta_p(g_1)\theta_q(g_2) \text{ and } pP \cap qP = pp'P, pp'=qq' \\ 0 &,\text {otherwise.} \end{cases}$
\item $\sum_{\overline{g} \in G/\theta_p(G)} e_{g,p} = 1 \quad \text{if } N_p < \infty$, 
\end{enumerate}
where $e_{g,p} \doteq u_g^{\phantom{*}}s_p^{\phantom{*}}s_p^*u_g^*$ and $N_p \doteq [G:\theta_p(G)]$. 

\bigskip 

As a cancellative abelian semigroup, $P$ embeds into its Grothendieck group $H\doteq P^{-1}P$, which  is obviously
countable and discrete.
As observed in \cite[Section 2]{ACRS} 
its Pontryagin dual $T$, which is a compact metrizable group,  acts on $\CQ(\gxp)$ by a gauge action $\gamma$ defined by $\gamma_\chi(u_gs_p) \doteq \chi(p) \ u_gs_p$. The corresponding fixed-point algebra 
 is denoted by  $\CF$ 
and coincides with  
$\overline{{\rm span}}\{ u_g^{\phantom{*}}s_p^{\phantom{*}}s_p^*u_h^* \mid g,h \in G, p \in P\}$. 
%
%
%

Denote by $P(\perp) \doteq \{ (p,q) \in P\times P\mid pP\cap qP = pqP\}$ the collection of all \emph{relatively prime pairs} in $P$. 
We define an equivalence relation $\sim$ on $P(\perp)\times P(\perp)$ 
 by saying that 
two pairs $(p,q),(\tilde{p},\tilde{q}) \in P(\perp)$ are equivalent whenever they satisfy $p^{-1}q = \tilde{p}^{-1}\tilde{q}$, or equivalently if there is an $x \in P^*$ such that $\tilde{p}=xp$ and $\tilde{q}=xq$. 
 For each $(p,q) \in P(\perp)$, we define a contractive linear map $F_{(p,q)}\colon \CQ(\gxp) \to \CF$ by $a \mapsto \int_T \gamma_\chi(s_p^{\phantom{*}}as_q^*) \ d\chi$, where the integration is with respect to the normalized Haar measure on the compact abelian group $T$. Given an element $a$ in $\CQ(\gxp)$, the terms $F_{(p,q)}(a)$ should be regarded as its Fourier coefficients
and, within this framework, one can show \cite[Section 4]{ACRS} that 
if the element actually sits in 
${\rm span}\{ u_g^{\phantom{*}}s_{p'}^{\phantom{*}}s_{q'}^*u_h^{\phantom{*}} \mid g,h \in G, p',q' \in P\}$, which is dense in $\CQ(\gxp)$, there is a uniquely determined finite set $A(a) \subset  P(\perp)/_\sim$ with the property that 
\begin{equation}\label{eq:F_(p,q) reconstruction formula}
\begin{array}{c} a = \sum\limits_{[(p,q)] \in A(a)} s_p^*F_{(p,q)}(a)s_q^{\phantom{*}}, \end{array}
\end{equation} 
In the next section we will show that for an arbitrary element of
 $\CQ(\gxp)$ 
a similar result 
holds.

\section{Main result}

Since P is an abelian semigroup, its Grothendieck group $H= P^{-1} P $ is also an abelian group. As such, it is in particular amenable. Now there are
many a way in which amenability can be expressed. One that is particularly suited to our context is exploiting the existence of the so-called F{\o}lner sequences.
A F{\o}lner sequence for a discrete group $H$ is a sequence $\{F_n: n\in\IN\}$ of finite subsets $F_n\subset H$ such that for every $h\in H$
$\lim_n \frac{|hF_n\triangle F_n|}{|F_n|}=0$, where $\triangle$ denotes the symmetric difference between two sets, i.e.
$A\triangle B\doteq (A\cup B)\setminus (A\cap B)$.   As can be proved, it is always possible to produce out of a given F{\o}lner sequence a new F{\o}lner sequence which is also increasing, that is
$F_n\subset F_{n+1}$ and exhausting, namely $\cup_n F_n=H$. Henceforth our F{\o}lner sequences will always be assumed  both increasing and
exhausting.\\

Associated with every F{\o}lner sequence, there is a sequence of real-valued functions defined as $\varphi_n(h)\doteq \frac{|hF_n\cap F_n|}{|F_n|}$, $h\in H$.
Their most relevant properties are summarized in a couple of easy yet useful lemmas. 

\begin{lemma}\label{support}
The functions $\varphi_n: H\rightarrow \IR$ are  finitely supported, and $\{{\rm supp}\,\varphi_n: n\in\IN\}$ is an increasing and exhausting family of subsets of $H$.
\end{lemma}

\begin{proof}
By definition, the support of $\varphi_n$ is the set $\{h\in H: hF_n\cap F_n\neq\emptyset\}$, which is easily seen to coincide with
$F_nF_n^{-1}$. Therefore, we have ${\rm supp}\,\varphi_n\subset {\rm supp}\,\varphi_{n+1}$ and $\cup_n{\rm supp}\, \varphi_n=H$.

\end{proof}

\begin{lemma}\label{defpos}
The functions $\varphi_n$ are all positive definite. Moreover, for every $h\in H$ $\lim_n \varphi_n(h)= 1$.
\end{lemma}
\begin{proof}
The statement about positive definiteness is clear, as $\varphi_n$ is nothing but $\frac{\langle\lambda_h \chi_{F_n}, \chi_{F_n}\rangle}{\langle\chi_{F_n}, \chi_{F_n}\rangle}$, where 
$\lambda: H\rightarrow \CU(\ell_2(H))$ is the left regular representation of $H$. Finally, the limit property comes from the  F{\o}lner property, i.e. 
$\lim_n\frac{|hF_n\triangle F_n|}{|F_n|}=0$ for every $h\in H$. 
Indeed, we certainly have $\limsup_n \varphi_n(h)\leq 1$.  On the other hand, we also have
$$
\liminf_n\varphi_n(h)=\liminf_n \frac{|hF_n\cap F_n|}{|F_n|}\geq\liminf_n \frac{|hF_n\cup F_n|}{|F_n|}-\lim_n \frac{|hF_n\triangle F_n|}{|F_n|}\geq 1
$$
This concludes the proof.
\end{proof}

As for every integer $n$ the support of the function $\varphi_n$ is finite, the sum
$$
S_n(x)\doteq \sum_{h=[(p,q)]\in H} \varphi_n(h) s_p^* F_{(p,q)}(x)s_q\quad \textrm{for every}\,x\in\CQ(\gxp)
$$
is well defined. We now want to prove  the above sums have good approximation properties.
To this aim, the above $S_n$'s are better understood as linear operators acting on the $C^*$-algebra $\CQ(\gxp)$, merely thought of as a Banach space w.r.t. its norm.

\begin{lemma}\label{bounded}
The linear operators $S_n$ are uniformly bounded. More precisely, we have $\|S_n\|=1$ for every $n\in\IN$.
\end{lemma}

\begin{proof}
Inserting the expression of the Fourier coefficients $ F_{(p,q)}(x)$ into the definition of the sums $S_n$, the following representation formula is easily arrived at
$$
S_n(x)=\int_T K_n(\chi)\gamma_\chi(x)\textrm{d}\chi,
$$
where the integral kernel $K_n$ is explicitly given by $K_n(\chi)=\sum_{h\in H} \varphi_n(h)\chi(h)$, and $\gamma_\chi$ is the gauge automorphism of $\CQ(\gxp)$ corresponding to the character $\chi\in T= \widehat{H}$.
Among  other things, the representation formula above  also has the merit of showing the sums $S_n$ do not depend on the representatives $(p,q)$ of $h=[(p,q)]$. 
The thesis will be fully  proved as soon as we show $\|K_n\|_{L^1(K)}=1$ for every $n\in\N$. This is done in two steps.
We first prove that $K_n(\chi)\geq 0$ for every $\chi\in T$. This is in turn achieved by an application of Bochner's theorem that a function is positive if and only if its Fourier transform is positive definite. In our case, this Fourier transform is easily computed. Indeed, 
$$
\widehat{K_n}(k)=\int_T K_n(\chi)\overline{\chi(k)}\textrm{d}\chi=\sum_{h\in H}\varphi_n(h)\int_T \chi(h)\overline{\chi(k)}\textrm{d}\chi=\varphi_n(k)
$$ 

Therefore, the functions $K_n$ are positive thanks to
Lemma \ref{defpos}. This allows us to compute their $L^1$-norms exactly. Indeed, we have
$$
\|K_n\|_{L^1(T)}=\int_T K_n(\chi)\textrm{d}\chi=\sum_{h\in H}\varphi_n(h)\int_T \chi(h)\textrm{d}\chi=\varphi_n(1_G)=1
$$
thanks to the fact that $\int_T\chi(h)\textrm{d}\chi=0$ unless $h=1_G$, in which case the integral is $1$. 
\end{proof}

We are now in a position to prove the main result of this paper.

\begin{thm}\label{maintheorem}
Let $\CQ(\gxp)$ be the boundary quotient $C^*$-algebra associated to an algebraic dynamical system $(G,P,\theta)$,
as at the beginning of Section \ref{prel}.
For every $x\in\CQ(\gxp)$ the sums $S_n(x)$ converge to $x$ in norm when $n$ goes to infinity.
\end{thm}

\begin{proof}
Phrased differently, all we have to prove is $\{S_n:n\in\IN\}$ converges strongly to the identity operator on $\CQ(\gxp)$. As the the sequence is bounded
by Lemma \ref{bounded}, it is enough to check the statement on a convenient dense
subspace of our $C^*$-algebra.
In our previous work  \cite{ACRS} it was shown that  for every $x$ in the norm-dense subalgebra 
$$\CA\doteq {\rm span}\{ u_g^{\phantom{*}}s_{p'}^{\phantom{*}}s_{q'}^*u_h^{\phantom{*}} \mid g,h \in G, p',q' \in P\}\subset\CQ(\gxp)
$$
there is a finite set $F_x\subset H$ such that $x=\sum_{h=[(p,q)]\in F_x} s_p^*F_{(p,q)}(x)s_q$
which is moreover uniquely determined under the requirement that $F_x \subset  P(\perp)/_\sim\subset H$. It is obvious that  for such a
$x$ the sequence $S_n(x)=\sum_{h=[(p,q)]\in H}\varphi_n(h)s_p^*F_{(p,q)}(x)s_q$ must converge to $x$, since
we will eventually have $F_x\subset{\rm supp}\,\varphi_n$  and $\varphi_n(h)\rightarrow 1$ by Lemmas \ref{support} and \ref{defpos} respectively. 
We can also be more explicit and display the usual underlying $\frac{\varepsilon}{3}$-argument in full detail.  So given
any $x\in\CQ(\gxp)$, pick an element $x_\varepsilon\in\CA$ such that $\|x-x_\varepsilon\|<\frac{\varepsilon}{3}$. Then we have
$$
\|x-S_n(x)\|\leq \|x-x_\varepsilon\|+\|x_\varepsilon-S_n(x_\varepsilon)\|+\|S_n(x_\varepsilon-x)\|<\frac{2\varepsilon}{3}+\|x_\varepsilon-S_n(x_\varepsilon)\|
$$
Since $S_n(x_\varepsilon)$ converges to $x_\varepsilon$, there exists $N_\varepsilon\in\IN$ such that  $\|S_n(x_\varepsilon)-x_\varepsilon\|<\frac{\varepsilon}{3}$ for every $n\geq N_{\varepsilon}$. This finally says that $\|S_n(x)-x\|<\varepsilon$ for every $n\geq N_\varepsilon$.
\end{proof}

\begin{example}
The $2$-adic ring $C^*$-algebra $\CQ_2$ is by definition  the universal $C^*$-algebra generated by a unitary $U$
and a (proper) isometry $S_2$ such that $S_2U=U^2S_2$ and $S_2S_2^*+US_2S_2^*U^*=1$ (for more information see e.g. \cite{LarsenLi, ACR, ACR2, ACR3, ACR4}).
The Cuntz algebra $\CO_2$, as is known, is the universal $C^*$-algebra generated by two isometries $S_1, S_2$ such that
$S_1S_1^*+S_2S_2^*=1$.  
It is rather obvious that $\CO_2$ embeds into $\CQ_2$. Indeed, there exists an injective $^*$-homomorphism that sends $S_1$ to $US_2$ and $S_2$ to $S_2$.
As shown in \cite{ACR}, $\CQ_2$ can also be realized as the boundary quotient $C^*$-algebra rising from
the algebraic dynamical system in which $P$ is the unital semigroup generated by $2$ in $\IN^\times$ acting on $Z$
by multiplication. Clearly, the Grothendieck group $H$ of $P$ is isomorphic with $\IZ$. 
If we choose the F{\o}lner sequence $\{F_n:n\in\IN\}$ given by $F_n\doteq[-n,n]$, then the functions $\varphi_n: \IZ\to \IR$ are seen at once to be $\varphi_n(i)=1-|i|/(2n+1)$, $-n \leq i \leq n$ and $\varphi_n(i)=0$ otherwise.
In this setting Theorem \ref{maintheorem} gives  the following formula, which holds in the norm topology
\begin{align*}
x & =\lim_n F_{(1,1)}(x)+\sum_{i=1}^n \left( 1-\frac{|i|}{2n+1}\right) (S_2^*)^i F_{(2^i,1)}(x)+\left( 1-\frac{|i|}{2n+1}\right)F_{(1,2^i)}(x)S_2^i\\
& = \lim_n \int_\IT \widetilde{\alpha_z}(x) \textrm{d}z +\sum_{i=1}^n \left( 1-\frac{|i|}{2n+1}\right)  \int_\IT (\widetilde{\alpha_z}(x)z^i + \widetilde{\alpha_z}(x)z^{-i}) \textrm{d}z 
\end{align*}
where $\textrm{d}z$ is the normalized Haar measure on the one-dimensional torus. Of course, in the previous formula the isometry $S_2$ may be replaced by $S_1$ and thus the limit may be expressed in terms of the maps $F_{i}: \CQ_2\to \CQ_2^\IT$ and $F_{-i}: \CQ_2\to \CQ_2^\IT$, with $i\geq 0$, defined in \cite[Section 3.3]{ACR} . 
We also note that when $x\in\CO_2$, we recover the well-known  result for the Cuntz algebra $\CO_2$ (see e.g. \cite{Power}). 
Finally, a similar formula could also be written for the $p$-adic ring $C^*$-algebra $\CQ_p$ for all $2\leq p<\infty$. \\
In  \cite{NEK} Nekrashevych found a natural embedding of the Thompson groups into $\CU(\CO_2)$, which in turn embeds into $\CU(\CQ_2)$.
The Thompson group $F$ is here represented as  the subgroup generated by the elements $x_0\doteq S_2^2S_2^*+S_2S_1(S_1S_2)^*+S_1(S_1^*)^2$ and $x_1\doteq S_2S_2^*+S_1S_2^2(S_1S_2)^*+S_1S_2S_1(S_1^2S_2)^*+S_1^2(S_1^*)^3$. 
We find it interesting to rewrite these elements in terms of the Fourier coefficients defined above. 
After some easy computations one gets
\begin{align*}
x_0 & = F_{(1,2)}(x_0)S_2+F_{(1,1)}(x_0)+S_2^*F_{(2,1)}(x_0)
\end{align*}
where $F_{(1,2)}(x_0)=S_2^2(S_2^2)^*$, $F_{(1,1)}(x_0)=S_2S_1(S_1S_2)^*$, $F_{(2,1)}(x_0)=S_2S_1(S_1^2)^*$, and
\begin{align*}
x_1 & = F_{(1,2)}(x_1)S_2+F_{(1,1)}(x_1)+S_2^*F_{(2,1)}(x_1)
\end{align*}
where $F_{(1,2)}(x_1)=S_1S_2^2(S_2S_1S_2)^*$, $F_{(1,1)}(x_1)=S_2S_2^*+S_1S_2S_1(S_1^2S_2)^*$, $F_{(2,1)}(x_1)=S_2S_1^2(S_2^3)^*$.
\end{example}
\begin{example}
As already observed in \cite{ACRS}, the well-known $C^*$-algebra $\CQ_\IN$ may be recovered as the boundary quotient $\CQ(\IZ\rtimes \IN^\times)$ where $P=\IN^\times$ acts on $\IZ$ by multiplication.
In this case the Grothendieck group $H$ is given by the positive rational integers $\IQ_+^*\cong \oplus_{i=1}^\infty \IZ$ and its dual is $T\cong\prod_{i=1}^\infty \IT$. A  typical F{\o}lner sequence for $H$ is given by 
$$
F_n \doteq \{ (x_i)\in H \; | \; -n \leq x_j \leq n \; \forall \; 1 \leq j \leq n , x_i=0 \; \forall \; i>n \}\; .
$$
We observe that if $h=(h_1, \ldots , h_k, 0, \ldots)$, $n>\max\{k, |h_1|, \ldots , |h_k|\}$, then 
$$
\varphi_n(h)= \frac{\prod_{i=1}^k (2n+1-|h_i|) }{(2n+1)^k}\; . 
$$ 
In this particular example, Theorem \ref{maintheorem} gives  the following slightly less usual formula, which holds for any $x\in\CQ_\IN$ 
\begin{align*}
x & = \lim_n \sum_{h\in H} \sum_{(f_i)\in F_n}  \varphi_n(h)  \int_T \left(\prod_{i=1}^\infty z_i^{f_i}\right) \alpha_{(z_k)}(x) \textrm{d}\mu
\end{align*}
where d$\mu$ is the normalized Haar measure on $T$, the sequence $(z_i)$ is an element in $T$ and $\alpha_{(z_i)}$ is the gauge automorphism mapping $s_p$ to $z_ps_p$, $p \in \IN^\times$. Of course the infinite product is well defined because $f_i$ is eventually zero.
\end{example}

\section{An application: the relative commutant of a family of generating isometries}

In \cite{ACRS} the relative commutant $C^*(\{s_p: p\in P\})'\cap\CQ(\gxp)$ has been proved to be as small
as possible, to wit it is nothing but the $C^*$-subalgebra generated by the set of unitaries $\{s_p:p\in P^*\}$, under the hypothesis that $P^*$ is finite.
Here $P^*\subset P$ is the group of all invertible elements of $P$.
As an application of the Fej\'er-type theorem obtained above, we now intend to show  that the finiteness assumption is actually unnecessary. 

\begin{thm}
If $\bigcap_{p \in P} \theta_p(G) = \{1_G\}$,
then the relative commutant $C^*(\{s_p: p\in P\})'\cap\CQ(\gxp)$ coincides with  $C^*(P^*)$.
In particular, the relative commutant  reduces to  the multiples of the identity if $P^*=\{1_P\}$.
\end{thm}
\begin{proof}
We only have to deal with the inclusion $C^*(\{s_p: p\in P\})'\cap\CQ(\gxp)\subset C^*(P^*)$, for the reverse inclusion is trivially checked. If $w\in \CQ(\gxp)$ is a unitary that commutes with  $s_p$ for every $p\in P$, then  $F_{(p,q)}(w)=0$ unless  $[(p,q)]\in P^*$, in which case the Fourier coefficient is just a scalar, as we had already proved in
\cite{ACRS}. For such a $w$ the sequence $\{S_n(w): n\in\IN\}$ is therefore all contained in $C^*(P^*)$, and thus $w$ is also an element of
$C^*(P^*)$ since it is the limit of the sequence $\frac{1}{N}\sum_{i=0}^{N-1}S_n(w)$. 
\end{proof}

\begin{remark}
It might be worth stressing that the above application really seems to be out of the reach 
of the Fej\'er-type theorems we already knew of before writing
the present note.
\end{remark}

We end with a result concerning the structure of $C^*(P^*)$.
As one would expect, $C^*(P^*)$ can be proved to be isomorphic with the group $C^*$-algebra $C^*_{\rm{red}}(P^*)$  in a  number of relevant cases. To this aim, we recall that for any discrete group $\Gamma$, 
 the reduced $C^*$-algebra of $\Gamma$, here denoted by  $C^*_{\rm{red}}(\Gamma)$, 
 is the $C^*$-subalgebra of $\CB(\ell_2(\Gamma))$ generated by the left regular representation of $\Gamma$. Phrased differently, 
$C^*_{\rm{red}}(\Gamma)$ is the concrete $C^*$-algebra generated by the set of unitaries $\{\lambda_\gamma: \gamma\in \Gamma\}$ acting on $\ell_2(\Gamma)$ as $\lambda_\gamma \delta_k\doteq \delta_{\gamma k}$, for any $k\in\Gamma$, where
$\{\delta_k:k\in\Gamma\}$ is the canonical basis of $\ell_2(\Gamma)$. We denote by $C^*_\pi(\Gamma)$ the full $C^*$-algebra
of $\Gamma$. This is the enveloping $C^*$-algebra of the convolution algebra $L^1(\Gamma)$, namely the completion of the $*$-Banach algebra $L^1(\Gamma)$ under the maximal $C^*$-norm. The map given by
$$
\ell_1(\Gamma)\ni \sum_{\gamma\in\Gamma}\alpha_\gamma\delta_\gamma  \rightarrow \sum_{\gamma\in\Gamma} \alpha_\gamma \lambda_\gamma\in C^*(\Gamma)_{\rm{red}}
$$ 
extends to a surjective $^*$-homomorphism $\pi:C^*_\pi(\Gamma)\rightarrow C^*_{\rm{red}}(\Gamma)$
by the very definition of the maximal $C^*$-norm. Furthermore,  the map is well known to be also injective if and only if
$\Gamma$ is amenable, which is certainly true when $\Gamma$ is abelian.\\

The cases where our claim on $C^*(P^*)$ holds are singled out by choosing a suitable class of algebraic dynamical systems. More precisely, we need the following condition to be fulfilled: there exists $g_0\in G$ such that $\theta_p(g_0)=g_0$ with $p\in P^*$ implies $p=1_P$. 
Because $g_0$ must not depend on $p\in P^*$, the condition is  actually stronger than merely requiring that $P^*\ni p\rightarrow\theta_p\in{\rm Hom}(G)$ is injective. However, it is satisfied in many of the examples we discussed in \cite{ACRS}. 
\begin{proposition}
The $C^*$-algebra $C^*(P^*)$ is isomorphic with $C^*_{\rm{red}}(P^*)$ if $P$ is abelian and $(G,P,\theta)$ satisfies the condition above.
\end{proposition}

\begin{proof}
 
The map $P^*\ni p\rightarrow s_p\in \CU(\ell_2(G))$ is a unitary representation of the discrete group $P^*$, therefore it can be lifted to a surjective  $^*$-homomorphism $\rho: C^*_\pi(P^*)\rightarrow C^*(P^*)$. Being abelian, $P^*$ is amenable as well,
hence the above homomorphism can be seen as an epimorphism from $C^*_{\rm{red}}(P^*)$   to $C^*(P^*)$, which   we still denote by $\rho$. The conclusion is immediately got to if $\rho$ is shown to be injective.
To this aim, we consider the isometry $W:\ell_2(P^*)\rightarrow \ell_2(G)$ that acts on the canonical basis of $\ell_2(P^*)$ as $W\delta_p\doteq \delta_{\theta_p(g_0)}$  for every $p\in P^*$, where $g_0\in G$ is such that $\theta_p(g_0)=g_0$ implies
$p=1_P$, which says $\theta_p(g_0)=\theta_q(g_0)$ is possible only when $p=q$.
Because the intertwining relation $W\lambda_p=s_p W$ is easily checked for every $p\in P^*$ (indeed, $W\lambda_p\delta_q=W\delta_{pq}=\delta_{\theta_{pq}(g_0)}=\delta_{\theta_p(\theta_q(g_0))}=s_pW\delta_q$), we see that   $\sum_{p\in F} \alpha_p\lambda_p= W^*( \sum_{p\in F} \alpha_p s_p)W$, for every finite set $F\subset P^*$, where the coefficients $\alpha_p$ are complex numbers. In particular, 
we find 
$$
\Big\|\sum_{p\in F} \alpha_p\lambda_p\Big\|= \Big\|W^*\Big( \sum_{p\in F} \alpha_p s_p\Big)W\Big\|\leq \Big\|\sum_{p\in F} \alpha_p s_p\Big\|
=\Big\|\rho\Big(\sum_{p\in F} \alpha_ps_p\Big)\Big\|
$$
Therefore, the inequality $\|x\|\leq \|\rho(x)\|$ holds for every $x\in C^*_{\rm{red}}(P^*)$ by density.
\end{proof}


\end{document}